\documentclass{amsart}

\usepackage{graphicx}

\newtheorem{theorem}{Theorem}[section]
\newtheorem{lemma}[theorem]{Lemma}
\newtheorem{proposition}[theorem]{Proposition}

\theoremstyle{definition}

\theoremstyle{remark}

\numberwithin{equation}{section}

\begin{document}

\title{Anti-symmetry of the second eigenfunction of the fractional Laplace operator in a 3-D ball}

\author{Rui A. C. Ferreira}
\address{ Grupo F\'isica-Matem\'atica,\\ Faculdade de Ci\^encias, Universidade de\\ Lisboa, Av. Prof. Gama Pinto, 2, 1649-003\\ Lisboa, Portugal.
   }
\email{raferreira@fc.ul.pt}
\thanks{The author was supported by the ``Funda\c{c}\~ao para a Ci\^encia e a Tecnologia (FCT)" through the program ``Investigador FCT" with reference IF/01345/2014.}

\subjclass[2000]{Primary 35P15, 35R11.}



\keywords{Fractional Laplacian, anti-symmetry, unit ball.}

\begin{abstract}
In this work we extend a recent result by Dyda \emph{et. al.} [B. Dyda, A. Kuznetsov, M. Kwa$\acute{\mbox{s}}$nicki, Eigenvalues of the fractional Laplace equation in the unit ball, J. Lond. Math. Soc. (2) {\bf 95} (2017), 500--518.] to dimension 3.
\end{abstract}

\maketitle

\section{Introduction and main result}

Let $d\geq 1$ and $D\subset\mathbb{R}^d$ be the unit ball. For $\alpha\in(0,2]$, define the fractional Laplace operator (see e.g. \cite{Bucur}) by (the case $\alpha=2$ is understood as the limiting case)

$$(-\Delta)^{\frac{\alpha}{2}}f(x)=-\frac{2^{\alpha}\Gamma\left(\frac{d+\alpha}{2}\right)}{\pi^{\frac{\alpha}{2}}\Gamma\left(-\frac{\alpha}{2}\right)}\lim_{\varepsilon\rightarrow 0}\int_{\mathbb{R}^d\backslash B(x,\varepsilon)}\frac{f(x)-f(y)}{|x-y|^{d+\alpha}}dy,$$
and consider the eigenvalue problem for $(-\Delta)^{\frac{\alpha}{2}}$ with a zero condition in the complement of $D$:
\begin{equation}\label{prob}
\left\{\begin{array}{ll}
(-\Delta)^{\frac{\alpha}{2}}\varphi(x)=\lambda\varphi(x),\quad & x\in D,\\
\varphi(x)=0,\quad & x\notin D.
\end{array}\right.
\end{equation}

This problem is being studied by several researchers in different directions but, here, we are interested in proving a result regarding the eigenfunctions of the second smallest eigenvalue of \eqref{prob}. More specifically, we will prove the following:

\begin{theorem}\label{mainresult}
Let $d=3$ and $0<\alpha\leq 2$. Let $\lambda$ be the second smallest eigenvalue of the problem \eqref{prob}. Then, the eigenfunctions corresponding to $\lambda$ are antisymmetric, i.e. they satisfy the relation $\varphi(-x)=-\varphi(x)$. 
\end{theorem}
The previous theorem was known for dimension $d=1$ and with restrictions on the parameter $\alpha$ (cf. \cite{Banuelos,Kwa}). Very recently it was proved in \cite{Dyda} for $\alpha=1$ and $1\leq d\leq 9$, or $0<\alpha\leq 2$ and $d\in\{1,2\}$. The proof is based on estimation of the eigenvalues of \eqref{prob}, in particular, on Theorem 1.2 proved therein. In \cite[Section 4.2]{Dyda} the authors reduce the proof of Theorem \ref{mainresult} (for $0<\alpha\leq 2$ and $d\in\{1,2\}$) into checking the truthfulness of two conditions (we will formally present these conditions in Section \ref{sec3}) and then, in the last paragraph of their text, they indicate that those conditions still hold when $d=3$ (being that based on numerical evidence), though they couldn't prove it. That is the objective of this manuscript.

In Section \ref{sec1} we introduce some notation and results that are to be used throughout. In Section \ref{sec3} we prove\footnote{All the symbolic computations within the proof were done using Maple Software.} Theorem \ref{mainresult} (we note that for dimension $d>3$ this is still an open problem).

\section{Some auxiliary results and notation}
\label{sec1}

In this section we present two results and introduce some notation used throughout this work.

We start with a result appearing in \cite[Theorem 2]{Dragomir}.

\begin{lemma}\label{in0}
Let $m,p,k\in\mathbb{R}$ with $m,p>0$ and $p>k>-m$. If
$$k(p-m-k)\geq (\leq) 0,$$
then
$$\Gamma(p)\Gamma(m)\geq(\leq)\Gamma(p-k)\Gamma(m+k).$$
\end{lemma}

The following monotonicity property was proved in \cite[Lemma 4.1]{Dyda}.

\begin{lemma}\label{inc}
The function
$$F(\alpha)=\frac{\Gamma(\alpha+3)\Gamma\left(\frac{\alpha}{2}+\frac{9}{2}\right)}{\Gamma\left(\frac{\alpha}{2}+2\right)\Gamma\left(\alpha+\frac{9}{2}\right)},$$
is increasing on $[0,2]$.
\end{lemma}

Now we need to introduce some functions (the notation used is the same as in \cite{Dyda}, considering $d=3$). For $n\in\mathbb{N}_0$ and $\alpha\in[0,2]$ we define:
\begin{align}
\mu_n(\alpha)&=\frac{2^\alpha\Gamma\left(\frac{\alpha}{2}+n+1\right)\Gamma\left(\frac{3+\alpha}{2}+n\right)}{n!\Gamma\left(\frac{3}{2}+n\right)}\label{mu},\\
\Lambda(\alpha)&=\frac{\mu_0(\alpha)\Gamma\left(\frac{\alpha}{2}+2\right)\Gamma\left(\frac{5}{2}+\alpha+2\right)(19\alpha+90)}{20\Gamma\left(\frac{5+\alpha}{2}+3\right)\Gamma\left(\alpha+2\right)}\label{Lambda},\\
a(\alpha)&=\frac{(14-3\alpha)(3+\alpha)}{1200(7+\alpha)}\nonumber,\\
b(\alpha)&=-\frac{1}{120}\frac{-\alpha^3+3\alpha^2+64\alpha+168}{7+\alpha}\nonumber,\\
T(\alpha)&=\frac{(90+19\alpha)\Gamma(\frac{\alpha}{2}+2)\Gamma(\alpha+\frac{9}{2})}{\Gamma(\frac{\alpha}{2}+\frac{11}{2})\Gamma(\alpha+2)}\label{T}.
\end{align}
Finally we define as usual the Psi-digamma function by
$$\Psi(x)=\frac{\Gamma'(x)}{\Gamma(x)},\quad x>0,$$
and the Psi-polygamma function by its derivatives $\Psi^{(n)}(x)$, $n\in\mathbb{N}$.

\section{Proof of Theorem \ref{mainresult}}\label{sec3}

According to the analysis done in \cite[Section 4]{Dyda} and subsequent sections, in order to prove Theorem \ref{mainresult}, it is sufficient to prove the next two conditions:

\begin{equation}\label{cond1}
\mu_2(\alpha)>\Lambda(\alpha),\quad\alpha\in(0,2],
\end{equation}
and
\begin{equation}\label{cond2}
g_\alpha(T(\alpha))<0,\quad\alpha\in(0,2],
\end{equation}
where $g_\alpha(t)=a(\alpha)t^2+b(\alpha)t+\alpha+2$.

We start with \eqref{cond1}, by inserting the corresponding definitions \eqref{mu} and \eqref{Lambda} to obtain
$$\frac{1}{120}\Gamma(\alpha+6)>\frac{4}{15}\frac{(90+19\alpha)\Gamma(\frac{\alpha}{2}+2)\Gamma(\alpha+\frac{9}{2})}{(\alpha+9)(\alpha+7)(\alpha+5)\Gamma(\frac{\alpha}{2}+\frac{3}{2})},$$
which is equivalent to
\begin{equation}\label{in10}
\frac{(\alpha+9)(\alpha+7)(\alpha+5)}{90+19\alpha}>32\frac{\Gamma(\frac{\alpha}{2}+2)\Gamma(\alpha+\frac{9}{2})}{\Gamma(\alpha+6)\Gamma(\frac{\alpha}{2}+\frac{3}{2})}.
\end{equation}
Now, in order to prove \eqref{in10}, note that
$$32\frac{\Gamma(\frac{\alpha}{2}+2)\Gamma(\alpha+\frac{9}{2})}{\Gamma(\alpha+6)\Gamma(\frac{\alpha}{2}+\frac{3}{2})}=4\frac{\alpha+7}{\alpha+4}\frac{\Gamma\left(\frac{\alpha}{2}+2\right)\Gamma\left(\alpha+\frac{9}{2}\right)}{
\Gamma(\alpha+3)\Gamma\left(\frac{\alpha}{2}+\frac{9}{2}\right)}\leq 2\frac{\alpha+7}{\alpha+4},$$
where we have used Lemma \ref{inc}. Therefore, once we show that
$$\frac{(\alpha+9)(\alpha+7)(\alpha+5)}{90+19\alpha}>2\frac{\alpha+7}{\alpha+4},\quad\alpha\in(0,2],$$
the veracity of \eqref{in10} will be confirmed. But it is elementary to verify the previous inequality as the left hand side is an increasing function, the right hand side is a decreasing function and both sides equal $\frac{7}{2}$ when $\alpha=0$. Condition \eqref{cond1} is, therefore, proved.

The next two propositions will prove the condition in \eqref{cond2}, which in turn conclude the proof of Theorem \ref{mainresult}.

\begin{proposition}\label{prop}
For all $\alpha\in[0,2]$ we have that
$$(90+19\alpha)\frac{2}{\alpha+9}\leq T(\alpha)\leq(90+19\alpha)\frac{\alpha+2}{\alpha+9}.$$
Moreover,
$$g_\alpha\left((90+19\alpha)\frac{2}{\alpha+9}\right)<0,\quad\alpha\in(0,2],$$
and
\begin{equation}
g_\alpha\left((90+19\alpha)\frac{\alpha+2}{\alpha+9}\right)<0,\quad\alpha\in(\alpha^\star,2],
\end{equation}
where
\begin{multline*}
\alpha^\star=\frac{1}{2679}
\sqrt[3]{118571508548+120555\sqrt{328018829721}}\\
+\frac{21023359}{2679\sqrt[3]{118571508548+120555\sqrt{328018829721}}}-\frac{8581}{2679}.
\end{multline*}
\end{proposition}

\begin{proof}
We start with the lower bound for $T$. By definition \eqref{T}, we have that
$$T(\alpha)=\frac{(90+19\alpha)\Gamma(\frac{\alpha}{2}+2)\Gamma(\alpha+\frac{9}{2})}{\Gamma(\frac{\alpha}{2}+\frac{11}{2})\Gamma(\alpha+2)}.$$
Now set $p=\frac{\alpha}{2}+2$, $k=-\frac{5}{2}$ and $m=\alpha+\frac{9}{2}$. Then, $p>k>-m$ and $k(p-m-k)\geq 0$ which, by Lemma \ref{in0}, imply that
$$\Gamma\left(\frac{\alpha}{2}+2\right)\Gamma\left(\alpha+\frac{9}{2}\right)\geq \Gamma\left(\frac{\alpha}{2}+\frac{9}{2}\right)\Gamma\left(\alpha+2\right),$$
which is equivalent to
$$\frac{\Gamma(\frac{\alpha}{2}+2)\Gamma(\alpha+\frac{9}{2})}{\Gamma(\frac{\alpha}{2}+\frac{11}{2})\Gamma(\alpha+2)}\geq\frac{2}{\alpha+9}.$$

Now, for the upper bound of $T$, we note that
$$T(\alpha)=\frac{(90+19\alpha)2(\alpha+2)\Gamma(\frac{\alpha}{2}+2)\Gamma(\alpha+\frac{9}{2})}{(\alpha+9)\Gamma(\frac{\alpha}{2}+\frac{9}{2})\Gamma(\alpha+3)},$$
and using Lemma \ref{inc} we immediately conclude that 
$$T(\alpha)\leq(90+19\alpha)\frac{\alpha+2}{\alpha+9}.$$

Let us now calculate
$$g_\alpha\left((90+19\alpha)\frac{2}{\alpha+9}\right)=\frac{1}{300}\alpha^2\frac{95\alpha^3+237\alpha^2-6300\alpha-26568}{(\alpha+7)(\alpha+9)^2},$$
which is negative on $\alpha\in(0,2]$ in virtue that $95\alpha^3+237\alpha^2-6300\alpha-26568<95\cdot 2^3+237\cdot 2^2-26568<0$.

Now,
$$g_\alpha\left((90+19\alpha)\frac{\alpha+2}{\alpha+9}\right)=-\frac{1}{1200}\alpha^2\frac{893\alpha^4+10367\alpha^3+36800\alpha^2+32472\alpha-13608}{(\alpha+7)(\alpha+9)^2}.$$
Since the derivative of $893\alpha^4+10367\alpha^3+36800\alpha^2+32472\alpha-13608$ is positive for all $\alpha>0$ and the polynomial has a zero on $[0,2]$ given by $\alpha^\star$, then
$$g_\alpha\left((90+19\alpha)\frac{\alpha+2}{\alpha+9}\right)<0,\quad\alpha\in(\alpha^\star,2],$$
and the proposition is proved.
\end{proof}
The previous result together with \eqref{cond1} show that Theorem \ref{mainresult} is proved on the interval $(\alpha^\star,2]$. The problem under consideration is harder to analyze when $\alpha$ is sufficiently small. Indeed, if one plot the graph of $T$ together with the (greatest) zero\footnote{It is easy to show that $b^2(\alpha)-4a(\alpha)(\alpha+2)\geq 0$} of $g$, i.e.
$$\frac{-b(\alpha)+\sqrt{b^2(\alpha)-4a(\alpha)(\alpha+2)}}{2a(\alpha)},$$
they are apparently indistinguishable when $\alpha$ is \emph{close} to zero. And, of course, we would like to prove that
\begin{equation}\label{in9}
T(\alpha)<\frac{-b(\alpha)+\sqrt{b^2(\alpha)-4a(\alpha)(\alpha+2)}}{2a(\alpha)},
\end{equation}
on $(0,\alpha^\star]$ which in turn would complete the proof of our main result, in virtue of Proposition \ref{prop}.

The only way we found out to remove the restriction on $\alpha$ and prove \eqref{cond2} was to consider not \eqref{in9} but the equivalent inequality:

\begin{equation}\label{in8}
\frac{\Gamma(\frac{\alpha}{2}+2)\Gamma(\alpha+\frac{9}{2})}{\Gamma(\frac{\alpha}{2}+\frac{11}{2})\Gamma(\alpha+2)}<\frac{-b(\alpha)+\sqrt{b^2(\alpha)-4a(\alpha)(\alpha+2)}}{2a(\alpha)(90+19\alpha)}.
\end{equation}
The following result completes the proof of Theorem \ref{mainresult}.

\begin{proposition}
Inequality \eqref{in8} holds for all $\alpha\in(0,\alpha^\star]$.
\end{proposition}
\begin{proof}
We start defining two functions on $[0,\alpha^\star]$:
$$f(\alpha)=\frac{\Gamma(\frac{\alpha}{2}+2)\Gamma(\alpha+\frac{9}{2})}{\Gamma(\frac{\alpha}{2}+\frac{11}{2})\Gamma(\alpha+2)},$$
and
$$h(\alpha)=\frac{-b(\alpha)+\sqrt{b^2(\alpha)-4a(\alpha)(\alpha+2)}}{2a(\alpha)(90+19\alpha)}.$$
Note that $f(0)=\frac{2}{9}=h(0)$. We will show that
\begin{equation}\label{inprinc}
f'(\alpha)\leq f'(0)<h'(0)\leq h'(\alpha),\quad \alpha\in[0,\alpha^\star],
\end{equation}
which immediately implies the inequality in \eqref{in8}. 

We start to calculate the derivative of $f$:
$$f'(\alpha)=\frac{\Gamma(\frac{\alpha}{2}+2)\Gamma(\alpha+\frac{9}{2})}{\Gamma(\frac{\alpha}{2}+\frac{11}{2})\Gamma(\alpha+2)}\left[\frac{1}{2}\Psi\left(\frac{\alpha}{2}+2\right)+\Psi\left(\alpha+\frac{9}{2}\right)-\frac{1}{2}\Psi\left(\frac{\alpha}{2}+\frac{11}{2}\right)-\Psi\left(\alpha+2\right)\right].$$
From the above expression we see that $f'(0)=\frac{671}{2835}-\frac{2}{9}\ln(2)$. Showing that $f'(\alpha)\leq f'(0)$ is equivalent to showing that
\begin{equation}\label{in7}
\frac{1}{2}\Psi\left(\frac{\alpha}{2}+2\right)+\Psi\left(\alpha+\frac{9}{2}\right)-\frac{1}{2}\Psi\left(\frac{\alpha}{2}+\frac{11}{2}\right)-\Psi\left(\alpha+2\right)\leq f'(0)\frac{\Gamma(\frac{\alpha}{2}+\frac{11}{2})\Gamma(\alpha+2)}{\Gamma(\frac{\alpha}{2}+2)\Gamma(\alpha+\frac{9}{2})}.
\end{equation}
Since, by Lemma \ref{inc}, the following inequality holds
$$f'(0)\frac{\Gamma(\frac{\alpha}{2}+\frac{11}{2})\Gamma(\alpha+2)}{\Gamma(\frac{\alpha}{2}+2)\Gamma(\alpha+\frac{9}{2})}\geq f'(0)\frac{\alpha+9}{\alpha+2},$$
then, to prove \eqref{in7}, it is sufficient to show that
$$\underbrace{\frac{1}{2}\Psi\left(\frac{\alpha}{2}+2\right)+\Psi\left(\alpha+\frac{9}{2}\right)-\frac{1}{2}\Psi\left(\frac{\alpha}{2}+\frac{11}{2}\right)-\Psi\left(\alpha+2\right)}_{=r(\alpha)}\leq \underbrace{f'(0)\frac{\alpha+9}{\alpha+2}}_{=s(\alpha)}.$$
But since $r(0)=s(0)$, then it is sufficient to prove that
\begin{equation}\label{in6}
r'(\alpha)\leq s'(\alpha),\quad \alpha\in[0,\alpha^\star].
\end{equation}
We have that
$$r'(\alpha)=\frac{1}{4}\Psi'\left(\frac{\alpha}{2}+2\right)+\Psi'\left(\alpha+\frac{9}{2}\right)-\frac{1}{4}\Psi'\left(\frac{\alpha}{2}+\frac{11}{2}\right)-\Psi'\left(\alpha+2\right),$$
and
$$s'(\alpha)=\frac{1}{405}\frac{-671+630\ln(2)}{(\alpha+2)^2}.$$
It is well known that 
$$\Psi^{(n)}(x)=(-1)^{n+1}n!\sum_{k=0}^\infty\frac{1}{(x+k)^{n+1}},$$
from where we infer that $\Psi'(x)>0$ and $\Psi'(x)$ is decreasing for all $x>0$. Therefore,
$$r'(\alpha)\leq\frac{1}{4}\Psi'\left(2\right)+\Psi'\left(\frac{9}{2}\right)-\frac{1}{4}\Psi'\left(\frac{\alpha^\star}{2}+\frac{11}{2}\right)-\Psi'\left(\alpha^\star+2\right)<-0.1795.$$
On the other hand, $s'$ is increasing, hence $s'(\alpha)\geq s'(0)>-0.1447$, which completes the proof of \eqref{in6}.

Let us now return to \eqref{inprinc} and prove that $h'(0)\leq h'(\alpha)$ for all $\alpha\in[0,\alpha^\star]$.
It is easy to check that $\alpha^4-6\alpha^3+25\alpha^2+1104\alpha+2944>0$ on $[0,\alpha^\star]$ and fastidious to check that:
$$h''(\alpha)=-20\frac{A(\alpha)}{B(\alpha)},$$
where ($x(\alpha)=\sqrt{\alpha^4-6\alpha^3+25\alpha^2+1104\alpha+2944}/(\alpha+7)$)
\begin{align*}
A(\alpha)&=x(\alpha)(9861\alpha^{11}+21506\alpha^{10}+230639\alpha^{9}+22832627\alpha^{8}+367617434\alpha^{7}
\\
&+5168851898\alpha^6+42711607466\alpha^5+202807642502 \alpha^4+576435831104\alpha^3\\
&+1071472458168\alpha^2+1458665131392\alpha+1171994600448)-9861\alpha^{12}+234441\alpha^{11}\\
&+4246101\alpha^{10}+10907731\alpha^9+210630942\alpha^8+4000462638\alpha^7+33568792782\alpha^6\\
&+263084581722\alpha^5+1699521987312\alpha^4+6179874342344\alpha^3+9813287816640\alpha^2\\
&+1154602149120\alpha-8833393336320,
\end{align*}
and
\begin{equation*}
B(\alpha)=(90+19\alpha)^3(3+\alpha)^3(-14+3\alpha)^3\sqrt{(\alpha^4-6\alpha^3+25\alpha^2+1104\alpha+2944)^3}.
\end{equation*}
Evidently $B(\alpha)<0$ for all $\alpha$. Now, regarding the sign of the function $A$, we note that there are only two negative terms, $-9861\alpha^{12}$ and $-8833393336320$, that in turn are not enough to make $A$ negative. To see that, firstly we note that $-9861\alpha^{12}+234441\alpha^{11}>0$. Secondly, we note that, though rather tedious to prove it, it is nevertheless true the following inequality,
$$1171994600448 x(\alpha)-8833393336320>0,$$
which proves that $A(\alpha)>0$. In conclusion $h''(\alpha)>0$ for all $\alpha\in[0,\alpha^\star]$.

Finally, since $f'(0)<0.083<h'(0)$, the proof is done.
\end{proof}

\section{acknowledgements}
The author would like to thank Professor Pedro Antunes for fruitful discussions on the problem. 

The author is also grateful for the kind email-reply of Professors Dyda, Kuznetsov and Kwa$\acute{\mbox{s}}$nicki.

\bibliographystyle{amsplain}

\begin{thebibliography}{10}

\bibitem{Banuelos}
R. Ba$\tilde{\mbox{n}}$uelos\ and\ T. Kulczycki, The Cauchy process and the Steklov problem, J. Funct. Anal. {\bf 211} (2004), no.~2, 355--423. 

\bibitem{Bucur}
C. Bucur\ and\ E. Valdinoci, {\it Nonlocal diffusion and applications}, Lecture Notes of the Unione Matematica Italiana, 20, Springer, 2016. 

\bibitem{Dragomir}
S. S. Dragomir, R. P. Agarwal\ and\ N. S. Barnett, Inequalities for beta and gamma functions via some classical and new integral inequalities, J. Inequal. Appl. {\bf 5} (2000), no.~2, 103--165.

\bibitem{Dyda}
B. Dyda, A. Kuznetsov, M. Kwa$\acute{\mbox{s}}$nicki, Eigenvalues of the fractional Laplace equation in the unit ball, J. Lond. Math. Soc. (2) {\bf 95} (2017), 500--518.

\bibitem{Kwa}
M. Kwa$\acute{\mbox{s}}$nicki, Eigenvalues of the fractional Laplace operator in the interval, J. Funct. Anal. {\bf 262} (2012), no.~5, 2379--2402.

\end{thebibliography}

\end{document}